\documentclass[amssymb,11pt]{amsart}
\usepackage{amscd,amssymb,euscript,amsthm}
\theoremstyle{plain}
\newtheorem{thm}{Theorem}[section] 
\newtheorem{cor}[thm]{Corollary}
\newtheorem{prop}[thm]{Proposition}
\newtheorem{lem}[thm]{Lemma}

\theoremstyle{definition}
\newtheorem{defn}[thm]{Definition}
\theoremstyle{remark}
\newtheorem{rem}[thm]{Remark}

\numberwithin{equation}{section}%assigns eqns section numbers
\newcommand{\id}{\operatorname{id}}

\newcommand{\rank}{\operatorname{rank}}

\newcommand{\Sep}{\operatorname{Sep}}
\def\<{\left<}
\def\>{\right>}
\def\cstar{$C^*$-algebra}
\begin{document}
\title[Quantum channels]{Quantum channels that\\
preserve entanglement}
\author{William Arveson}
%
%\thanks{*supported by 
%NSF grant DMS-0100487} 
%
%\address{Department of Mathematics,
%University of California, Berkeley, CA 94720}
%
%\email{arveson@math.berkeley.edu}
%
%\subjclass{46L55, 46L09}
\date{16 January, 2008}

\begin{abstract} Let $M$ and $N$ be full matrix algebras.  
A unital completely positive (UCP) map $\phi:M\to N$ is said to {\em preserve entanglement} 
if its inflation $\phi\otimes \id_N : M\otimes N\to N\otimes N$ has the following property: 
for every maximally entangled pure state $\rho$ of $N\otimes N$, $\rho\circ(\phi\otimes \id_N)$ 
is an entangled state of $M\otimes N$.  

We show that there is a dichotomy in 
that every UCP map that is not entanglement breaking in the sense of 
Horodecki-Shor-Ruskai must preserve entanglement, and that entanglement preserving 
maps of every possible rank exist in abundance.  We also show that with 
probability $1$, {\em all} UCP maps 
of relatively small rank preserve entanglement, but that this is not so for UCP maps 
of maximum rank.  
\end{abstract}
\maketitle

\section{Introduction}\label{S:in}

Let $H$ and $K$ be finite dimensional Hilbert spaces.  In the literature 
of quantum information theory, a {\em quantum channel} 
(from $\mathcal B(H)$ to $\mathcal B(K)$) can be described equivalently as a completely positive linear map 
\begin{equation}\label{inEq0}
\psi: \mathcal B(H)^\prime\to \mathcal B(K)^\prime
\end{equation}
from 
the dual of $\mathcal B(H)$ to the dual of $\mathcal B(K)$ that carries states to states.  
One can view quantum channels as 
the morphisms of a category whose objects are the dual spaces $\mathcal B(H)^\prime$ of 
finite dimensional type I factors $\mathcal B(H)$.  Quantum channels are 
the {\em adjoints} of unital completely positive (UCP) maps in the sense that the most general 
map $\psi$ of  (\ref{inEq0}) must have the form 
\begin{equation*}
\psi(\rho)=\rho\circ\phi, \qquad \rho\in \mathcal B(H)^\prime, 
\end{equation*}
where $\phi:\mathcal B(K)\to \mathcal B(H)$ is a UCP map.  
In this paper we focus on UCP maps, keeping in 
mind that all statements about the category of UCP maps (with objects $\mathcal B(H)$)
translate contravariantly 
into statements about the category of quantum channels (with objects $\mathcal B(H)^\prime$).  

A state $\rho$ of $\mathcal B(K\otimes H)$ is called {\em separable} 
if it is a convex combination 
of product states 
\begin{equation}\label{inEq1}
\rho(a\otimes b)=\sum_{k=1}^s t_k \cdot\sigma_k(a)\tau_k(b), \qquad a\in \mathcal B(K), \quad b\in \mathcal B(H), 
\end{equation}
where $\sigma_k$ and $\tau_k$ are states of $\mathcal B(K)$ and $\mathcal B(H)$ respectively, 
and the $t_k$ are positive numbers with sum $1$.   States that are not separable are 
said to be {\em entangled}.  
Since the set of all separable states of $\mathcal B(K\otimes H)$ is 
compact (see Remark 1.1 of \cite{arvEnt1}), the entangled states form a relatively open 
subset of the state space of $\mathcal B(K\otimes H)$.

Since the tensor product of two completely positive maps is completely positive, 
every UCP map $\phi: \mathcal B(K)\to \mathcal B(H)$ gives rise to an inflated UCP 
map $\phi\otimes\id: \mathcal B(K\otimes H)\to \mathcal B(H\otimes H)$, 
defined uniquely by sending $a\otimes b$ to $\phi(a)\otimes b$, $a\in \mathcal B(K)$, $b\in\mathcal B(H)$.  
In turn,  $\phi\otimes\id$ induces a map from states $\rho$ of $\mathcal B(H\otimes H)$ 
to states $\rho^\prime=\rho\circ(\phi\otimes\id)$ of $\mathcal B(K\otimes H)$: 
\begin{equation}\label{inEq00}
\rho^\prime(a\otimes b)=\rho(\phi(a)\otimes b),\qquad a\in \mathcal B(K), \quad b\in \mathcal B(H).  
\end{equation}

The notion of an  entanglement breaking channel was introduced and studied 
in the papers (\cite{HorShRu} and \cite{rus}). 
In our context, a UCP map $\phi: \mathcal B(K)\to \mathcal B(H)$ is said to be 
{\em entanglement breaking} iff for 
every state $\rho$ of  $\mathcal B(H\otimes H)$, the state 
$\rho^\prime=\rho\circ(\phi\otimes \id)$ is a separable state of 
$\mathcal B(K\otimes H)$.  
It was pointed out that 
entanglement breaking UCP maps are the most 
degenerate, where in this case ``degeneracy" means that the associated quantum channel 
can be simulated by a classical channel.  That is because, 
as shown in \cite{HorShRu},  
the entanglement breaking 
UCP maps $\phi: \mathcal B(K)\to \mathcal B(H)$ 
are precisely those that admit a representation of the form 
\begin{equation}\label{inEq3}
\phi(x)=\sum_{k=1}^s \omega_k(x) e_k, \qquad x\in\mathcal B(K), 
\end{equation}
where $\omega_1,\dots,\omega_s$ are states of $\mathcal B(K)$ and $e_1,\dots,e_s$ are positive 
operators in $\mathcal B(H)$ having sum $\mathbf 1$.

We now introduce a class of UCP maps that appear to lie 
at the opposite extreme from the entanglement breaking 
ones.  Fix a UCP map $\phi:\mathcal B(K)\to \mathcal B(H)$ and consider 
the action of the channel  $\phi\otimes\id$ 
on {\em pure} states $\rho$ of $\mathcal B(H\otimes H)$.  If 
$x\in\mathcal B(H\otimes H)\mapsto \langle x\xi,\xi\rangle$ 
is the pure state  corresponding to a unit vector $\xi\in H\otimes H$,  then 
the corresponding state $\rho_\xi$ of $\mathcal B(K\otimes H)$ defined by  (\ref{inEq00}) becomes 
\begin{equation}\label{inEq4}
\rho_\xi(a\otimes b)=\langle (\phi(a)\otimes b)\xi,\xi\rangle, \qquad a\in\mathcal B(K), \quad 
b\in \mathcal B(H).  
\end{equation}
Notice that whenever $\xi=\eta\otimes \zeta$ decomposes into a tensor product of vectors in $H$,  $\rho_\xi$ 
decomposes into a tensor product of states.  In order to rule out such ``classical" 
correlations in pure states, we fix attention 
on unit vectors $\xi\in H\otimes H$ that are {\em marginally cyclic} in the sense that 
they satisfy 
\begin{equation}\label{inEq5}
(\mathcal B(H)\otimes\mathbf 1)\xi = H\otimes H, 
\end{equation}
or equivalently (see Remark \ref{mfRem1}), for every $b\in \mathcal B(H)$ one has 
$$
(\mathbf 1\otimes b)\xi=0\implies b=0.  
$$
Note that the second assertion is simply that the state of 
$\mathcal B(H)$ defined by $\omega(b)=\langle (\mathbf 1\otimes b)\xi,\xi\rangle$ 
should be {\em faithful}: $\omega(b^*b)=0\implies b=0$.  

\begin{defn}\label{inDef1}
 A UCP map $\phi:\mathcal B(K)\to \mathcal B(H)$ is said to  
{\em preserve entanglement} if for every marginally cyclic unit vector $\xi\in H\otimes H$, 
the state $\rho_\xi$ of (\ref{inEq4}) is an entangled state of $\mathcal B(K\otimes H)$.  
\end{defn}

In Section \ref{S:pu} we show how the parameterization of states given in \cite{arvEnt1} can be 
appropriately adapted 
to UCP maps so as to make the space $\Phi^r$ of all UCP maps $\phi:\mathcal B(K)\to \mathcal B(H)$ 
of rank $\leq r$ into a compact probability space that carries a unique unbiased probability measure 
$P^r$, and we show in Section \ref{S:ep} 
that $P^r$ is concentrated on the set of maps of rank $r$.  Thus, the probability space 
$(\Phi^r,P^r)$ represents {\em choosing a UCP map of rank $r$ at random}.  
We prove a zero-one law for channels in Section \ref{S:ee} which expresses in 
strong probabilistic terms the dichotomy  
that a UCP map either preserves 
entanglement or it has the degenerate form (\ref{inEq3}).  

We then apply the main results 
of \cite{arvEnt1} to show that there are plenty of entanglement preserving UCP maps 
of every possible rank, and that {\em almost surely every} UCP map of relatively small rank 
preserves entanglement (see Theorem \ref{epThm1}).    We conclude with a discussion of 
extreme points of the convex set of 
UCP maps that implies: Whenever an extremal UCP map of rank $r$ exists, then 
almost surely every UCP map of rank $r$ is extremal.

\begin{rem}[Terminology: maximally entangled pure states]\label{mfRem1}  Let $H$ be a finite dimensional 
Hilbert space.  We offer some remarks to support  
our singling out of marginally cyclic vectors as 
candidates for ``maximally entangled" pure states of $\mathcal B(H\otimes H)$.  
While the term {\em maximally entangled state} occurs frequently 
in the literature of physics and quantum information theory - particularly in the 
context of pure states - and while there are examples that appear to attach meaning 
to some deeper concept underlying the term, there appears to be no 
satisfactory mathematical definition of 
that concept.  We believe that the pure states associated with 
marginally cyclic vectors of $H\otimes H$ are an appropriate 
precise formulation of 
this undefined term because of the following observations.  

Let $\xi$ be a unit vector in $H\otimes H$ and let $\rho(x)=\langle x\xi,\xi\rangle$ 
be the corresponding pure state of $\mathcal B(H\otimes H)$.  $\rho$ restricts to 
a state $\omega$ 
of $\mathcal B(H)$, 
$$
\omega(b)=\langle (\mathbf 1\otimes b)\xi,\xi\rangle, \qquad b\in \mathcal B(H),   
$$
and to some extent, properties of the ``marginal" state $\omega$ determine properties 
of $\rho$.  
For example, $\rho$ is separable in 
the sense that $\xi$ can be decomposed into a tensor product $\eta\otimes \zeta$ with 
$\eta,\zeta\in H$ iff $\omega$ is a {\em pure} state of $\mathcal B(H)$.  

Since pure states are those whose density operators have rank $1$, this suggests that 
in order for $\rho$ to be ``maximally entangled", it is reasonable to require that the rank 
of the density operator of $\omega$ should be as large as possible; i.e., that 
$\omega$ should be a {\em faithful state} of $\mathcal B(H)$.  If we accept 
that as a definition,  then 
$\rho(x)=\langle x\xi,\xi\rangle$ is considered to be maximally entangled whenever 
$$
\omega(b^*b)=0\implies b=0, \qquad b\in\mathcal B(H),   
$$
or equivalently,  for every $b\in\mathcal B(H)$, 
\begin{equation}\label{mfEq2}
(\mathbf 1\otimes b)\xi=0 \implies b=0.  
\end{equation}
In turn, since the two von Neumann algebras 
$\mathcal B(H)\otimes\mathbf 1$ and $\mathbf 1\otimes \mathcal B(H)$ 
are commutants of each other, (\ref{mfEq2}) is equivalent to the assertion 
\begin{equation*}
\{(a\otimes \mathbf 1)\xi: a\in\mathcal B(H)\}=H\otimes H.      
\end{equation*}
We have found it useful to regard a unit vector $\xi\in H\otimes H$ 
as ``maximally entangled" precisely when it is marginally cyclic.  In any case, 
the familiar examples, such as $\xi=n^{-1/2}(e_1\otimes e_1+\cdots+e_n\otimes e_n)$ 
where $\{e_k\}$ is an orthonormal basis for $H$, are all marginally cyclic vectors.  
\end{rem}

Since writing this paper, we learned from M. B. Ruskai that a definition of 
``maximally entangled pure state" has been proposed in \cite{cub1} and \cite{cub2} that is 
equivalent to the one we have proposed above (such vectors are said to have ``maximum Schmidt rank" 
in \cite{cub1} and \cite{cub2}).  Basically, those authors obtain information about the 
relations between 
subspaces $M\subseteq H\otimes H$ with the property 
that every unit vector in $M$ has ``Schmidt rank" at least $r$ and they apply 
their results to some of the measures of entanglement that have been proposed 
in the literature of quantum information theory.  

\section{Real-analytic parameters for UCP maps}\label{S:pu}

Let $H$ and $K$ be finite dimensional Hilbert spaces with $n=\dim H$, $m=\dim K$ 
and fix a UCP map $\phi: \mathcal B(K)\to \mathcal B(H)$.  
A straightforward application of Stinespring's 
theorem (as formulated in Appendix \ref{S:ap}) implies that there is an $r$-tuple of operators $v_1,\dots,v_r\in\mathcal B(H,K)$ such that 
\begin{equation}\label{puEq1}
\phi(a)=v_1^*av_1+\cdots+v_r^*av_r, \qquad a\in\mathcal B(K), 
\end{equation}
and that the operators $v_k$ satisfy 
\begin{equation}\label{puEq3}
v_1^*v_1+\cdots+v_r^*v_r=\mathbf 1_H.  
\end{equation}
Moreover, one 
can arrange that $v_1,\dots, v_r$ are linearly independent, and in that case 
the integer $r$ is called the {\em rank} of $\phi$.  
Let $\Phi^r(K,H)$ be the compact space of all UCP maps 
$\phi: \mathcal B(K)\to \mathcal B(H)$ of rank 
at most $r$.  Since $H$ and $K$ will be held fixed, we lighten notation 
by writing $\Phi^r$ for $\Phi^r(K,H)$.    

In this section we show that 
for every $r=1,2,\dots,mn$, there is a convenient parameterization of 
the space $\Phi^r$ and we describe its basic properties.  
While this is a reformulation of some of the results of \cite{arvEnt1}, there are 
enough differences in the two formulations that it is appropriate to discuss this 
parameterization of $\Phi^r$ in some detail.  

Given two $r$-tuples $(v_1,\dots,v_r)$ and $(v_1^\prime,\dots, v_r^\prime)$ 
of operators  in 
$\mathcal B(H,K)$ which 
are {\em not} necessarily linearly independent, then by Proposition \ref{apProp1} of 
Appendix \ref{S:ap}, 
$$
v_1^*xv_1+\cdots+v_r^*xv_r=v_1^{\prime *}xv_1^\prime+\cdots+v_r^{\prime *}xv_r^\prime, 
\qquad x\in\mathcal B(K)
$$
 iff there is a unitary $r\times r$ matrix $(\lambda_{ij})\in M_r(\mathbb C)$ 
such that 
\begin{equation}\label{puEq2}
v_i^\prime=\sum_{j=1}^r \lambda_{ij}v_j, \qquad i=1,2,\dots, r.  
\end{equation}

Now consider the space $V^r(H,K)$ of all $r$-tuples 
$v=(v_1,\dots,v_r)$ with operator components $v_k\in\mathcal B(H,K)$ that satisfy 
$v_1^*v_1+\cdots+v_r^*v_r=\mathbf 1_H$ (we do {\em not} require that the component operators 
are linearly independent).  Theorem 2.1 of \cite{arvEnt1} implies that 
 $V^r(H,K)$ is a compact connected 
real-analytic Riemannian manifold 
that is acted upon transitively by a compact group of isometries.  
For every $v=(v_1,\dots,v_r)\in V^r(H,K)$, 
\begin{equation}\label{puEq3.1}
\phi_v(x)=v_1^*xv_1+\cdots + v_r^*xv_r, \qquad x\in \mathcal B(K)
\end{equation}
defines a UCP map $\phi_v:\mathcal B(K)\to \mathcal B(H)$ of rank at most $r$.  
The following result summarizes the 
properties of this parameterization $v\mapsto\phi_v$ and is a direct consequence 
of the preceding remarks.  

\begin{prop}\label{puProp1}
Fix two finite dimensional Hilbert spaces $H$, $K$ with $\dim H=n$, $\dim K=m$.  
For every $r=1,2,\dots,mn$, let $\Phi^r$ be the compact space of all UCP maps 
$\phi:\mathcal B(K)\to \mathcal B(H)$ of rank $\leq r$.

Every element of $\Phi^r$ has the form (\ref{puEq3.1}) for some 
$v\in V^r(H,K)$.  This parameterization $v\mapsto \phi_v$ is continuous 
and one has $\phi_v=\phi_{v^\prime}$ iff $v$ and $v^\prime$ belong to the 
same $U(r)$ orbit as in (\ref{puEq2}).  Hence the map $v\mapsto\phi_v$ 
promotes uniquely to a homeomorphism of the 
orbit space $V^r(H,K)/U(r)$ onto the space $\Phi^r$ of UCP maps of rank $\leq r$.  
\end{prop}

Fixing $H$, $K$ as in Proposition \ref{puProp1}, consider the integer $q=m^2n^2+1$, and the much larger 
group $U(q)$ of all unitary $q\times q$ matrices $(\lambda_{ij})\in M_q(\mathbb C)$.    We single out 
the following subset of $V^r(H,K)$, 
$$
\Sep(V^r(H,K))=\bigcup_{\lambda\in U(q)} Z_\lambda
$$
where 
$$
Z_\lambda=\{w=(w_1,\dots,w_r)\in V^r(H,K): \rank(\sum_{j=1}^r\lambda_{ij}w_j)\leq 1, \quad 1\leq i\leq q\}.  
$$

The key property of $\Sep(V^r(H,K))$ is described as follows.  

\begin{prop}\label{puProp2}
Let $\phi:\mathcal B(K)\to \mathcal B(H)$ be a UCP map of rank $r$ and let 
$\xi\in H\otimes H$ be a marginally cyclic unit vector.  Then 
the state $\rho_\xi$ of $\mathcal B(K\otimes H)$ defined by 
$$
\rho_\xi(a\otimes b)=\langle (\phi(a)\otimes b)\xi,\xi\rangle, \qquad a\in\mathcal B(K), 
\quad b\in \mathcal B(H)
$$
is separable iff $v\in\Sep(V^r(H,K))$.  
\end{prop}

\begin{proof}
This is a restatement of Proposition 7.7 of \cite{arvEnt1}.  
\end{proof}

After noting that the condition $v\in \Sep(V^r(H,K))$ does not depend on the choice 
of marginally cyclic vector $\xi$, we can combine Proposition \ref{puProp2} 
with a result of \cite{HorShRu} to conclude:

\begin{cor}\label{puCor1}  Let $\phi: \mathcal B(K)\to \mathcal B(H)$ be an arbitrary UCP map
 and let 
$\mathcal S_\phi$ be the set of all states $\rho$ of $\mathcal B(K\otimes H)$ of 
the form 
\begin{equation}\label{puEq4}
\rho(a\otimes b)=\langle(\phi(a)\otimes b)\xi,\xi\rangle, \qquad a\in\mathcal B(K), \quad 
b\in\mathcal B(H)
\end{equation}
where $\xi$ ranges over the set of marginally cyclic unit vectors in $H\otimes H$.  
If $\mathcal S_\phi$ contains a single entangled state then every state of $\mathcal S_\phi$ is 
entangled and $\phi$ preserves entanglement.  Otherwise, $\phi$ is entanglement 
breaking.    
\end{cor}

\begin{proof}
To prove the last sentence, let $e_1,\dots,e_n$ be an orthonormal 
basis for $H$ and let $\xi$ be the marginally cyclic unit vector 
$$
\xi=\frac{1}{\sqrt n}(e_1\otimes e_1+\cdots+e_n\otimes e_n).  
$$
The implications B $\iff$ C of Theorem 4 of \cite{HorShRu} are equivalent to the 
assertion that $\phi$ is entanglement breaking iff the state $\rho$ is separable, 
hence the assertion follows from the first two sentences of Corollary \ref{puCor1}.    
\end{proof}

\section{A zero-one law for UCP maps}\label{S:ee}

The unit sphere $S(H)=\{\xi\in H: \|\xi\|=1\}$ of an $n$ dimensional Hilbert space $H$ 
is $S^{2n-1}$, a real-analytic Riemannian symmetric space that carries a unique unitarily 
invariant probability measure $\mu_H$.  

Fix a UCP map $\phi:\mathcal B(K)\to \mathcal B(H)$.  Every vector $\xi$ in 
the unit sphere of $H\otimes H$ gives rise to a state $\rho_\xi$ of $\mathcal B(K\otimes H)$
by way of 
\begin{equation}\label{eeEq1}
\rho_\xi(a\otimes b)=\langle (\phi(a)\otimes b)\xi,\xi\rangle, \qquad a\in\mathcal B(K), 
\quad b\in\mathcal B(H), 
\end{equation}
thereby obtaining a map $\hat\phi:\xi\mapsto \rho_\xi$ from $S(H\otimes H)$ to 
states of $\mathcal B(K\otimes H)$ that we can view as a random variable associated 
with the probability space $(S(H\otimes H), \mu_{H\otimes H})$.  
We now show 
that it is possible to determine whether $\phi$ preserves entanglement in a way that makes no 
reference to marginally cyclic vectors, but rather to properties of the random 
variable $\hat\phi$.  Indeed, Theorem \ref{eeThm1} frames 
the dichotomy of UCP maps as follows: 
{\em  The channel associated with a given UCP map $\phi: \mathcal B(K)\to \mathcal B(H)$
either maps almost all pure states to separable states, or it maps almost all pure states 
to entangled states.}  That assertion is expressed more concisely as a zero-one law: 
$$
\mu_{H\otimes H}\{\xi\in S(H\otimes H): \hat\phi(\xi) {\rm{\ is\ entangled\,}}\}=0\ {\rm or\ } 1.  
$$

\begin{thm}\label{eeThm1}
For every UCP map $\phi: \mathcal B(K)\to \mathcal B(H)$, the following are equivalent:
\begin{enumerate}
\item[(i)]  For almost every vector $\xi\in S(H\otimes H)$, the state $\rho_\xi$ 
is entangled.    
\item[(ii)]  For every vector $\xi$ in some Borel subset of 
$S(H\otimes H)$ of positive measure, the state $\rho_\xi$ of 
(\ref{eeEq1}) is entangled.  
\item[(iii)]  For every marginally cyclic unit vector $\xi\in H\otimes H$, the state $\rho_\xi$ 
is entangled; i.e., $\phi$ preserves entanglement.  
\item[(iv)]  There exists a marginally cyclic unit vector $\xi\in H\otimes H$ such that $\rho_\xi$ is entangled.  
\end{enumerate}
\end{thm}

The proof of Theorem \ref{eeThm1} requires: 

\begin{lem}\label{tcLem1}
The set of marginally cyclic unit vectors of $H\otimes H$ is relatively open and dense in 
$S(H\otimes H)$,  and its complement has measure zero. 
\end{lem}

\begin{proof}  Let $Z$ be the set of all vectors $\xi\in S(H\otimes H)$ that 
are {\em not} marginally cyclic.  Since $\mu_{H\otimes H}$ assigns positive mass 
to nonempty open subsets of the unit sphere, it suffices  
to show that $Z$ is a closed set of $\mu_{H\otimes H}$-measure zero.  
Since the unit sphere $S(H\otimes H)$ is a connected real-analytic submanifold of its 
ambient space $H\otimes H$,  for every 
real-analytic function 
\begin{equation}\label{tcEq2}
F: S(H\otimes H)\to W
\end{equation}
that takes values in a finite dimensional 
real vector space $W$, either $F$ vanishes identically or the set of zeros of $F$ is a 
closed set 
of $\mu_{H\otimes H}$-measure zero (see Proposition B.1 of \cite{arvEnt1}).  
Thus, in order to show that $\mu_{H\otimes H}(Z)=0$, it suffices to exhibit a real-analytic 
function $F$ as in (\ref{tcEq2}) that does not vanish identically 
on $S(H\otimes H)$ such that  $Z=\{\xi\in S(H\otimes H): F(\xi)= 0\}$.  

We exhibit 
such a function $F$ as follows.  
We view $H\otimes H$ as $\mathbb C^n\otimes H$, where 
$n= \dim H$,  in which case $\mathbb C^n\otimes H$ is identified with the direct sum 
of $n$ copies of $H$, $\mathbf 1_{\mathbb C^n}\otimes\mathcal B(H)$ is identified 
with $n\times n$ diagonal operator matrices $(b_{ij})$ with $b_{11}=\cdots=b_{nn}\in\mathcal B(H)$, 
 and its commutant $\mathcal B(\mathbb C^n)\otimes\mathbf 1_H$ is identified 
with the set of all $n\times n$ operator matrices with 
entries in $\mathbb C\cdot \mathbf 1_H$.  

Let $\xi=(\xi_1,\dots,\xi_n)$ be a unit vector in $\mathbb C^n\otimes H$.  
Viewing $\xi$ as a column vector, straightforward 
verification shows that $\xi$ is marginally cyclic, i.e., 
$(\mathcal B(\mathbb C^n)\otimes \mathbf 1)\xi=\mathbb C^n\otimes H$,  
iff its components satisfy 
\begin{equation*}
{\rm {span\,}}\{\xi_1,\dots,\xi_n\} = H,   
\end{equation*}
or equivalently, iff $\{\xi_1,\dots,\xi_n\}$ is linearly independent.  

Consider the function $F: \mathbb C^n\otimes H\to \wedge^nH=H\wedge\cdots\wedge H$ defined by 
$$
F(\xi_1,\dots,\xi_n)=\xi_1\wedge\xi_2\wedge\cdots\wedge\xi_n.  
$$
$F$ is a homogeneous polynomial of degree $n$, and elementary multilinear algebra 
shows that for every $(\xi_1,\dots,\xi_n)\in\mathbb C^n\otimes H$, the components $\xi_k$ form a linearly 
independent set iff $\xi_1\wedge\cdots\wedge\xi_n\neq 0$.  Hence the restriction 
of $F$ to the unit sphere $S(\mathbb C^n\otimes H)$ is a real-analytic function 
with the property $Z=\{\bar \xi\in S(\mathbb C^n\otimes H): F(\bar \xi)=0\}$.  
Obviously, $F$ does not vanish identically on $S(\mathbb C^n\otimes H)$, since it 
is nonzero on any $n$-tuple $\bar\xi=(\xi_1,\dots,\xi_n)$ with linearly independent components $\xi_k$.  
\end{proof}

\begin{proof}[Proof of Theorem \ref{eeThm1}]
Since (i)$\implies$(ii) is trivial and  (iv)$\iff$(iii) follows 
from Corollary \ref{puCor1}, it suffices to prove  
(ii)$\implies$(iv) and (iii)$\implies$(i).  

(ii)$\implies$(iv): Let $A$ be a Borel subset of $S(H\otimes H)$ of positive measure such 
that $\rho_\xi$ is entangled for every $\xi\in A$.  
By Lemma \ref{tcLem1}, the set $M$  of marginally cyclic vectors in $S(H\otimes H)$ is 
an open dense set whose complement has measure zero.  Hence $M\cap A$ must have positive 
measure, and is therefore nonempty.  (iv) follows.

(iii)$\implies$(i):   This follows from another application of Lemma \ref{tcLem1}.
\end{proof}

\section{abundance of entanglement preserving maps}\label{S:ep}

Throughout this section, $H$ and $K$ denote Hilbert spaces of respective finite dimensions 
$n$ and $m$, and for the main results below we require that $n\leq m$.  
Proposition \ref{puProp1} asserts that for every $r= 1, 2,\dots, mn$, the map 
$$
v\in V^r(H,K)\mapsto \phi_v\in \Phi^r
$$ 
promotes to a homeomorphism of the orbit space $V^r(H,K)/U(r)$ onto the compact 
space $\Phi^r$ of all UCP maps $\phi: \mathcal B(K)\to \mathcal B(H)$ of rank $\leq r$.  
The unique invariant probability measure $\mu$ of $V^r(H,K)$ promotes to 
a probability measure $P^r$ on $\Phi^r$, defined on Borel sets $E\subseteq \Phi^r$ by 
$$
P^r(E)=\mu\{v\in V^r(H,K): \phi_v\in E\},  
$$
and Theorem 3.3 of \cite{arvEnt1} is equivalent to the following key assertion about this 
nonatomic topological 
probability space $(\Phi^r, P^r)$: 

\begin{thm}\label{epThm0}
For each $r=1,\dots,mn$, the measure $P^r$ is 
concentrated on the relatively 
closed subset of $\Phi^r$ consisting of UCP maps of rank $=r$. 
\end{thm}
We conclude that 
{\em for every $r=1,2,\dots, mn$, the probability space $(\Phi^r, P^r)$ 
represents ``choosing a UCP map of rank $r$ at random"}.  

In the following result 
we convert the principal results of \cite{arvEnt1} into assertions about 
the entanglement 
preserving part of $(\Phi^r,P^r)$.  
We write $EP(\Phi^r)$ for the set of 
all entanglement preserving maps in $\Phi^r$. 

\begin{thm}\label{epThm1} Let $H$, $K$ satisfy 
$n=\dim H\leq m=\dim K<\infty$.  
\begin{enumerate}
\item[(i)] For every $r=1,2,\dots,mn$, $EP(\Phi^r)$ is a relatively open subset of 
$\Phi^r$ of positive measure.  
\item[(ii)] For every $r$ satisfying $1\leq r\leq n/2$,  $P^r(EP(\Phi^r))=1$.  
\item[(iii)] For the maximum rank $r=mn$ one has $0<P^{mn}(EP(\Phi^{mn}))<1$.
\end{enumerate}
\end{thm}

The proof requires some material from \cite{arvEnt1}, which we summarize 
for the reader's convenience.

\begin{rem}[Subvarieties of $V^r(H,K)$] \label{epRem1}
By a {\em subvariety} of 
$V^r(H,K)$ we mean a subset of the form $Z=\{v\in V^r(H,K): F(v)=0\}$, where 
$$
F: V^r(H,K)\to W
$$
is a real-analytic function taking values in a finite dimensional real vector space 
$W$.  Let $\mu$ be the unique probability measure on $V^r(H,K)$ that is invariant 
under the transitive action by isometries.  Proposition 2.6 of \cite{arvEnt1} asserts 
that {\em every proper subvariety $Z\neq V^r(H,K)$ has $\mu$-measure zero. } 
\end{rem}

\begin{rem}[The wedge invariant]\label{epRem2}
In \cite{arvEnt1} we introduced an invariant of states called the wedge 
invariant.  The wedge invariant can be interpreted as a pair of random variables on the 
probability space $(V^r(H,K),\mu)$ as follows.  Every $r$-tuple 
$v=(v_1,\dots,v_r)\in V^r(H,K)$ gives rise to an operator $v_1\wedge\cdots\wedge v_r$ from 
$\otimes^rH$ to $\otimes^r K$ as in (1.5) of \cite{arvEnt1}, and $v_1\wedge\cdots\wedge v_r$ 
maps the symmetric 
subspace $\otimes^rH_+$ of $\otimes^rH$ to the antisymmetric subspace $\wedge^rK$ 
 of $\otimes^rK$.  Similarly, 
$v_1^*\wedge\cdots\wedge v_r^*$ maps the symmetric subspace of $\otimes^rK$ to the 
antisymmetric subspace of $\otimes^rH$.  Thus we obtain a pair of integer-valued 
random variables $w(\cdot)$, $w^*(\cdot)$ defined on $V^r(H,K)$ by 
$$
w(v)=\rank(v_1\wedge\cdots\wedge v_r\restriction_{\otimes^rH_+})
\quad
w^*(v)=\rank(v_1^*\wedge\cdots\wedge v_r^*\restriction_{\otimes^rK_+}).  
$$
These functions $w(\cdot)$ and $w^*(\cdot)$ are associated with subvarieties 
as follows.  
Propositions 8.1 and 8.2 of \cite{arvEnt1} imply that for every $r=1,\dots,mn$,  
$$
A=\{v\in V^r(H,K): w(v)\leq 1\} , \quad A^*=\{v\in V^r(H,K): w^*(v)\leq 1\} 
$$
are subvarieties of $V^r(H,K)$ and that 
$\Sep(V^r(H,K))\subseteq A\cap A^*$.  
\end{rem}

\begin{proof}[Proof of Theorem \ref{epThm1}]  (i): Combining the 
discussion preceding Theorem \ref{epThm1} with the discussion of 
Section \ref{S:pu}, one sees that the parameterization map 
$v\mapsto \phi_v$ gives rise to a measure preserving surjection of topological 
probability spaces $(V^r(H,K), \mu)\to (\Phi^r, P^r)$, which carries the 
closed set $\Sep(V^r(H,K))$ to 
the space of entanglement breaking maps of $\Phi^r$ and 
carries its complement to $EP(\Phi^r)$.  
Hence the assertion (i) is that for every $r$ one has $\mu(\Sep(V^r(H,K))<1$, which follows from 
Theorem 7.8 of \cite{arvEnt1}.  

(ii):  We have seen that $EP(V^r(H,K))=V^r(H,K)\setminus\Sep(V^r(H,K))$
is an open set.  We make use of the random variable of Remark \ref{epRem2} 
$$
w^*:V^r(H,K)\to \mathbb Z_+
$$ 
as follows.  
By Remark \ref{epRem2} above, 
$\Sep(V^r(H,K))\subseteq A^*$ and $A^*$ is a subvariety of $V^r(H,K)$.  
The critical fact is that since  
$r$ does not exceed $n/2$, 
Proposition  8.3 of \cite{arvEnt1} implies that $A^*$ is a {\em proper} subvariety 
of $V^r(H,K)$, and therefore has $\mu$-measure zero.  
Hence $\mu(\Sep(V^r(H,K))=0$.  
 
(iii): The remark following Proposition \ref{puProp2} makes it clear  
that Theorem 10.1 of \cite{arvEnt1} is equivalent to the assertion that $EP(\Phi^{mn})$ 
is a relatively open subset of $\Phi^r$ for which $0<P^r(EP(\Phi^r))<1$.  
\end{proof}

\begin{rem}[Estimating the critical rank]
Item (i) of Theorem \ref{epThm1} asserts that 
there are plenty of entanglement preserving UCP maps of 
every possible rank.  (ii) asserts that  essentially {\em all} UCP maps 
of relatively small rank must preserve entanglement, while (iii) implies that 
this breaks down for maps of maximum rank.  Hence 
there is a critical rank $r_0\leq mn$ with the property that essentially all UCP maps 
of rank $< r_0$ preserve entanglement, while $0<P^{r_0}(EP(\Phi^{r_0}))<1$.   
As we have pointed out in the context of states 
in Remark 12.3 of \cite{arvEnt1}, both bounds 
$n/2 < r_0\leq mn$ that follow directly from Theorem \ref{epThm1} 
seem overly conservative, and one would hope to have considerably 
more information about the size of $r_0$ in the future. 
\end{rem}

\section{abundance of extremals}\label{S:ae}
In this section we continue in the context of UCP maps $\phi: \mathcal B(K)\to \mathcal B(H)$ 
where $n=\dim H\leq \dim K=m<\infty$.  
In \cite{rusProb}, it was shown (in its dual form) that the extremal 
UCP maps $\phi: \mathcal B(H)\to \mathcal B(H)$ are dense in the set of all UCP maps of rank 
at most $n$, generalizing a result of \cite{rusSzWe} for $2\times 2$ matrix algebras. 
Our final result makes essentially the following assertion about extreme points of the convex 
set of UCP maps $\phi: \mathcal B(K)\to \mathcal B(H)$: {\em If there is an extremal UCP map 
of rank $r$, then almost surely every UCP map of rank $r$ is extremal. } 

\begin{thm}\label{aeThm1}
For every integer $r$ satisfying $1\leq r\leq n$, 
the extremals of rank $r$ in $(\Phi^r,P^r)$ are a relatively 
open dense set having probability $1$. 
There are no extremal UCP maps $\phi:\mathcal B(K)\to \mathcal B(H)$ of rank $>n$.  
\end{thm}

The proof of Theorem \ref{aeThm1} requires

\begin{lem}\label{aeLem1}Let $r$ be an integer satisfying $1\leq r\leq n\leq m$.  Then 
there is an $r$-tuple  $v=(v_1,\dots,v_r)\in V^r(H,K)$ such that 
$\{v_i^*v_j: 1\leq i,j\leq r\}$ is a linearly independent subset of $\mathcal B(H)$.  
\end{lem}

\begin{proof}
Let $e_1,\dots,e_r$ be an orthonormal set in $H$, let $p$ be the projection onto the linear 
span of $e_1,\dots,e_r$ and let $f$ be a unit vector in $K$.  
For each $i=1,\dots,r$ let $u_i$ be the rank-one partial isometry $u_i\xi=\langle\xi,e_i\rangle f$.  
Note that $\{u_i^*u_j: 1\leq i,j\leq r\}$ defines a system of matrix units for which 
$u_1^*u_1+\cdots+u_r^*u_r=p$, and in particular, $\{u_i^*u_j:1\leq i,j\leq r\}$ is a linearly independent 
subset of $\mathcal B(H)$.  

The rank of $p^\perp$ is $n-r\leq n-1\leq m-1$, hence  there is a projection $q\in\mathcal B(K)$ 
with $\rank q=\rank p^\perp$ whose range is orthogonal to $f$.  
Let $w$ be a partial isometry in $\mathcal B(H,K)$ having initial 
projection $p^\perp$ and final projection $q$, and set 
$$
v_i=u_i+r^{-1/2}w, \qquad i=1,2,\cdots,r.  
$$
One finds that $v_i^*v_j=u_i^*u_j+ r^{-1}p^\perp$, hence $v_1^*v_1+\cdots+v_r^*v_r=\mathbf 1_H$, 
and the set of all $v_i^*v_j=u_i^*u_j\oplus r^{-1}p^\perp$ is obviously linearly independent.  
\end{proof}

\begin{rem}\label{aeRem1}
Note that for any set of operators $v_1, \dots, v_r\in \mathcal B(H,K)$ for which 
$\{v_i^*v_j: 1\leq i,j\leq r\}$ is linearly independent, $\{v_1,\dots,v_r\}$ must 
be linearly independent.  For if $\lambda_i\in \mathbb C$ such that 
$\lambda_1\cdot v_1+\cdots+\lambda_r\cdot v_r=0$, then $\sum_{ij}\bar\lambda_i\lambda_j\cdot v_i^*v_j=0$, 
hence $\bar\lambda_i\lambda_j=0$ for all $i,j$, hence $|\lambda_i|^2=0$ for all $i$.  
\end{rem}

\begin{proof}[Proof of Theorem \ref{aeThm1}]
Consider the complex vector space 
$$
W=\wedge^{r^2}\mathcal B(H)=\mathcal B(H)\wedge\cdots\wedge \mathcal B(H), 
$$
the exterior product of $r^2$ copies of $\mathcal B(H)$, and 
let $F:V^r(H,K)\to W$ be the function obtained by restricting the function 
$$
v=(v_1,\dots,v_r)\in \mathcal B(H,K)^r\mapsto \bigwedge_{1\leq i,j\leq r} v_i^*v_j\in W
$$
to $V^r(H,K)$.  Since the above function is a real-homogeneous polynomial 
of degree $2r$, its restriction to $V^r(H,K)$ is real-analytic.  Moreover, Lemma 
\ref{aeLem1} implies that there is a point $v\in V^r(H,K)$ for which $F(v)\neq 0$.  
It follows that the set $Z=\{v\in V^r(H,K): F(v)=0\}$ of zeros of $F$ is a 
{\em proper} subvariety and therefore has measure zero and 
empty interior (see Remark \ref{epRem1}).  
By Remark \ref{aeRem1} and the remarks following (\ref{ap2Eq2}), for every $v\in V^r(H,K)$, the associated 
UCP map $\phi_v$ is extremal of rank $r$ iff $v\notin Z$.  This proves that the set of extremals 
of rank $r$ in $\Phi^r$ is an open dense subset whose complement has measure zero.    

The second sentence follows from the fact that if $\phi(x)=v_1^*xv_1+\cdots+v_r^*xv_r$ is extremal 
of rank $r$, then by the  remarks following (\ref{ap2Eq2}), the set of 
$r^2$ operators $\{v_i^*v_j: 1\leq i,j\leq r\}$ in $\mathcal B(H)$ is linearly independent.   
Since $\dim \mathcal B(H)=n^2$, it follows that $r^2\leq n^2$, hence $r\leq n$.  
\end{proof}

\appendix
\section{Remarks on Stinespring's theorem}\label{S:ap}

Stinespring's theorem (Theorem 1 of \cite{stine})  
provides a familiar and useful representation of completely 
positive maps.  Along with the existence of this representation there are 
notions of {\em minimality} and {\em uniqueness} - 
both of which have 
significant consequences, though neither minimality nor uniqueness was mentioned in 
the original source \cite{stine}.  We briefly review these facts here since 
we shall have to make use of all of them in this paper, 
referring the reader to pp 143-146 of \cite{arvSubalgI} for more detail.  

Let $A$ be a unital \cstar\ and let $\phi: A\to \mathcal B(H)$ be an 
operator-valued completely positive linear map.  The principal assertion
of Stinespring's theorem is that there is a pair $(\pi,V)$ consisting 
of a representation $\pi$ of $A$ on another 
Hilbert space $K$ and an operator $V: H\to K$ such that 
\begin{equation}\label{apEq1}
\phi(x)=V^*\pi(x)V,\qquad x\in A.  
\end{equation}
Such a pair $(\pi,V)$ will be called a {\em Stinespring pair} for $\phi$.  Two Stinespring 
pairs $(\pi_1, V_1)$ and $(\pi_1,V_2)$ are said to be {\em equivalent} if there is 
a unitary operator $U: K_1\to K_2$ such that 
\begin{equation}\label{apEq3}
UV_1=V_2, \quad {\rm and }\quad  U\pi_1(x)=\pi_2(x)U, \qquad x\in A.  
\end{equation}

A Stinespring pair $(\pi,V)$ is 
said to be {\em minimal} if $VH$ is a cyclic subspace for the 
representation $\pi$ in the sense that 
\begin{equation}\label{apEq2}
K=\overline{{\rm{span}}}\,\{\pi(x)V\xi: x\in A, \xi\in H\}.  
\end{equation}
The requirement (\ref{apEq2}) is equivalent to the following assertion about the relation 
of the subspace $VH$ to the commutant $\pi(A)^\prime$:
\begin{equation}\label{apEq4}
\forall \ \ b\in \pi(A)^\prime, \quad b\restriction_{VH}=0\implies b=0.  
\end{equation}
Every Stinespring pair $(\pi,V)$ can be reduced to a minimal one by replacing 
$\pi$ with the subrepresentation obtained by restricting $\pi$ to the 
reducing subspace of $K$ defined by the right side of (\ref{apEq2}).  
The uniqueness assertion is simply that {\em any two  minimal Stinespring 
pairs for $\phi$ are equivalent. }  

The immediate consequences of these results for UCP maps 
$\phi:\mathcal B(H_1)\to \mathcal B(H_2)$ between finite dimensional type I factors 
are as follows.  Taking $A=\mathcal B(H_1)$ and noting that the most general 
finite dimensional representation of $\mathcal B(H_1)$ is unitarily equivalent to a direct sum 
of $r=1,2,\dots$ copies of the identity representation, we conclude that there is 
a minimal Stinespring pair of the form $(\pi,V)$  where $\pi$ is the 
representation on $r\cdot H_1$ defined by 
\begin{equation}\label{apEq6}
\pi(x)=
\begin{pmatrix}
x&0&\cdots&0\\
0&x&\cdots&0\\
\vdots&\vdots&\vdots&\vdots\\
0&0&\cdots&x
\end{pmatrix}
,\qquad  x\in \mathcal B(H_1)  
\end{equation}
and where $V:H_2\to r\cdot H_1=H_1\oplus\cdots\oplus H_1$ 
is a linear map from $H_2$ to a direct sum of $r$ copies of $H_1$.  
The operator $V:H_2\to r\cdot H_1$ must have the form $V\xi=(v_1\xi,\dots,v_r\xi)$, $\xi\in H_2$ (viewed 
as a column vector), where 
$v_1,\dots,v_r$ is a uniquely determined $r$-tuple of operators in $\mathcal B(H_2,H_1)$.  
After these adjustments, the formula $\phi(x)=V^*\pi(x)V$ becomes 
\begin{equation}\label{apEq5}
\phi(x)=v_1^*xv_1+\cdots+v_r^*xv_r, \qquad x\in \mathcal B(H_1).  
\end{equation}
Since the commutant of $\pi(\mathcal B(H_1))$ consists of all 
$r\times r$ operator matrices with entries in $\mathbb C\cdot \mathbf 1_{H_1}$, the 
equivalence of (\ref{apEq2}) and (\ref{apEq4}) implies that the minimality 
of $(\pi,V)$  becomes this assertion: For every $\lambda_1,\dots,\lambda_r\in \mathbb C$ 
$$
\lambda_1v_1+\cdots+\lambda_r\cdot v_r=0 \implies \lambda_1=\cdots=\lambda_r=0, 
$$
i.e., iff the set of operators $\{v_1,\dots,v_r\}$ that implements (\ref{apEq5}) 
should be {\em linearly independent}.
In particular, these remarks show that 
the integer $r$  is uniquely defined by the formula (\ref{apEq5}) when 
$v_1,\dots,v_r$ is linearly independent; $r$ is called the {\em rank} of the completely positive map 
$\phi:\mathcal B(H_1)\to \mathcal B(H_2)$.  

The $r$-tuple $(v_1,\dots,v_r)$ that implements (\ref{apEq5}) is certainly not unique;   
but if $(v_1^\prime, \dots,v_r^\prime)$ is another 
such $r$-tuple, then the operator $V^\prime:H_2\to r\cdot H_1$ defined by 
$$
V^\prime\xi=(v_1^\prime\xi,\dots,v_r^\prime\xi), \qquad \xi\in H_2
$$ 
defines another Stinespring 
pair $(\pi,V^\prime)$ associated with the same representation of (\ref{apEq6}).  After recalling 
the structure of the commutant of $\pi(\mathcal B(H_1))$, one can apply 
the uniqueness assertion of Stinespring's theorem to conclude that there is 
a unique unitary matrix of scalars $(\lambda_{ij})\in U(r)$ such 
that 
\begin{equation}\label{spEq6}
v_i^\prime = \sum_{j=1}^r \lambda_{ij}\cdot v_j, \qquad i=1,2,\dots,r.  
\end{equation}

We require the following somewhat stronger 
form of uniqueness - in which the hypothesis of linear independence is dropped - that 
is valid when $H_1$ and $H_2$ are finite dimensional.  Notice however that {\em its proof is 
fundamentally the same as the proof of the preceding uniqueness assertion.}  Note too the 
resemblance between this result and 
Proposition 5.1 of \cite{arvEnt1},  which 
characterizes the possible representations 
of finite sums of positive rank one Hilbert space operators.  Indeed, though we do not require 
the fact,  there is a common generalization of both assertions to Hilbert $C^*$-modules.  

\begin{prop}\label{apProp1}
Let $(v_1,\dots,v_r)$ and $(v_1^\prime, \dots, v_r^\prime)$ be two $r$-tuples of operators 
in $\mathcal B(H_2, H_1)$.  Then one has 
\begin{equation}\label{apEq7}
v_1^*xv_1+\cdots+v_r^*xv_r=v_1^{\prime *}xv_1^\prime+\cdots +v_r^{\prime *}xv_r^\prime 
\end{equation}
for all $x\in \mathcal B(H_1)$ 
iff there is a unitary $r\times r$ matrix $(\lambda_{ij})\in U(r)$ that 
relates $(v_1^\prime,\dots,v_r^\prime)$ to $(v_1,\dots,v_r)$ as in  (\ref{spEq6}).   
\end{prop}

\begin{proof}
Assuming that (\ref{apEq7}) is satisfied, 
consider the completely positive map of $\mathcal B(H_1)$ to $\mathcal B(H_2)$ 
defined by $\phi(x)=v_1^*xv_1+\cdots+v_rxv_r^*$.  Let $\pi$ be the representation of 
$\mathcal B(H_1)$ on $r\cdot H_2$ defined by (\ref{apEq6}), and let $V, V^\prime$ be 
the two linear maps of $H_2$ into $r\cdot H_1$ defined by 
$V\xi=(v_1\xi,\dots,v_r\xi)$, $V^\prime\xi=(v_1^\prime\xi,\dots,v_r^\prime\xi)$.   
Notice that both $(\pi, V)$ and $(\pi, V^\prime)$ are (perhaps nonminimal) Stinespring pairs for $\phi$.  
Letting $K$ and $K^\prime$ be the linear subspaces of $r\cdot H_2$ spanned by 
$\pi(\mathcal B(H_1)VH_2$ and $\pi(\mathcal B(H_1)V^\prime H_2$ respectively, we find that 
that for all $x,y\in\mathcal B(H_1)$ and all $\xi,\eta\in H_2$
$$
\langle \pi(x)V\xi,\pi(y)V\eta\rangle=\langle \phi(y^*x)\xi,\eta\rangle=
\langle \pi(x)V^\prime\xi,\pi(y)V^\prime\eta\rangle.  
$$
A familiar argument shows that 
there is a unique isometry $U_0$ from $K$ to $K^\prime$ that satisfies 
$U_0\pi(x)V\xi=\pi(x)V^\prime\xi$ for all $x\in\mathcal B(H_1)$, $\xi\in H_2$, 
and one checks that 
$U_0$ intertwines the subrepresentations of $\pi$ defined by $K$ and $K^\prime$.  
Since the commutant of $\pi(\mathcal B(H_1))$ is a finite dimensional factor, $U_0$ can be 
extended to a unitary operator $U$ in the commutant of $\pi(\mathcal B(H_1))$.  
After noting that $U$ must be an operator matrix with scalar 
entries $(\lambda_{ij}\mathbf 1_{H_1})$, $\lambda_{ij}\in \mathbb C$, 
and noting that $UV=U\pi(\mathbf 1)V=\pi(\mathbf 1)V^\prime=V^\prime$, one obtains 
the desired unitary matrix $(\lambda_{ij})\in U(r)$.  
The proof of the converse is left for the reader.  
\end{proof}

\section{Remarks on extremal UCP maps}\label{S:ap2}

Let $A$ be a unital \cstar.  The extremal UCP maps from $A$ to $\mathcal B(H)$ were 
first determined in Theorem 1.4.6 of \cite{arvSubalgI}, which makes the following 
assertion in that case.  

\begin{thm}\label{ap2Thm1}
For every UCP map $\phi:A\to \mathcal B(H)$, the following are equivalent.  
\begin{enumerate}
\item[(i)]  $\phi$ is an extreme point of the convex set of all UCP maps from $A$ to $\mathcal B(H)$.  
\item[(ii)]  Let $(\pi,V)$ be a minimal Stinespring pair for $\phi$.  Then for every operator 
$b$ in the commutant of $\pi(A)$, 
\begin{equation}\label{ap2Eq1}
V^*bV=0\implies b=0.  
\end{equation}
\end{enumerate}
\end{thm}

Notice that in general, the condition (\ref{ap2Eq1}) is {\em stronger} than the condition 
(\ref{apEq4}) for minimality.  
Now specialize to the case in which $H_1$ and $H_2$ are finite dimensional Hilbert spaces 
and $\phi:\mathcal B(H_1)\to \mathcal B(H_2)$ is a UCP map.  Choosing an $r$-tuple of operators 
$v_1,\dots,v_r\in\mathcal B(H_2,H_1)$ as in (\ref{apEq5})
$$
\phi(x)=v_1^*xv_1+\cdots+v_r^*xv_r, \qquad x\in\mathcal B(H_1)
$$
and letting $\pi$ be the representation of $\mathcal B(H_1)$ on $r\cdot H_1$ 
defined in (\ref{apEq6}), we obtain 
a Stinespring pair $(\pi,V)$ for $\phi$ by defining $V:H_2\to r\cdot H_1$ 
as in Appendix \ref{S:ap}, viewing $V\xi=(v_1 \xi,\dots,v_r\xi)$ for $\xi\in H_2$ as a column vector 
with components in $H_1$.  Noting the structure of the 
commutant of $\pi(\mathcal B(H_1))$ pointed out in Appendix \ref{S:ap} following (\ref{apEq5}), one finds that 
the condition (\ref{ap2Eq1}) for extremality  becomes this: for every $r\times r$ matrix 
of scalars $(\lambda_{ij})$ 
\begin{equation}\label{ap2Eq2}
\sum_{i,j=1}^r\lambda_{ij}\cdot v_i^*v_j=0\implies \lambda_{ij}=0, 1\leq i,j\leq r, 
\end{equation}
and from Theorem \ref{ap2Thm1} we conclude 
that {\em $\phi$ is extremal iff the set of operators $\{v_i^*v_j: 1\leq i,j\leq r\}$ is 
linearly independent}.  The latter result is known as Choi's theorem in the quantum information
theory literature, which cites \cite{ChoiMat} as the source.

\subsection*{Acknowledgements}  I thank Mary Beth Ruskai for helpful remarks 
concerning material in \cite{HorShRu} and for pointing out 
several references.

%\vfill

\bibliographystyle{alpha}
\newcommand{\etalchar}[1]{$^{#1}$}
\newcommand{\noopsort}[1]{} \newcommand{\printfirst}[2]{#1}
  \newcommand{\singleletter}[1]{#1} \newcommand{\switchargs}[2]{#2#1}

\end{document}